\documentclass{amsart}
\usepackage{style}
\usepackage{mdframed}
\usepackage{longtable}
\usepackage{tikz-cd}

\begin{document}
	
	\title[Algebraic properties of face algebras]{Algebraic properties of face algebras}
	
	\author{Fabio Calder\'{o}n and Chelsea Walton}
	
	\address{Calder\'{o}n: Department of Mathematics,  Universidad Nacional de Colombia, Bogot\'{a}, Colombia}
	\email{facalderonm@unal.edu.co}
	
	\address{Walton: Department of Mathematics, Rice University, Houston, TX 77005, USA}
	\email{notlaw@rice.edu}
	
	\begin{abstract} 
	   Prompted an inquiry of Manin on whether a coacting Hopf-type structure $H$ and an algebra $A$ that is coacted upon share algebraic properties, we study the particular case of $A$ being a path algebra $\kk Q$ of a finite quiver $Q$ and $H$ being Hayashi's face algebra $\hay$ attached to $Q$. This is motivated by the work of Huang, Wicks, Won, and the second author, where it was established that the weak bialgebra coacting universally on $\kk Q$ (either from the left, right, or both sides compatibly) is $\hay$. For our study, we define the Kronecker square $\widehat{Q}$ of $Q$,  and show that $\hay \cong \kk \widehat{Q}$ as unital graded algebras. Then we obtain ring-theoretic and homological properties of $\hay$ in terms of graph-theoretic properties of $Q$ by way of $\widehat{Q}$.
	\end{abstract}
	
	\subjclass[2020]{16T20, 05C25}
	\keywords{face algebra, path algebra, quiver, Kronecker square}

	\maketitle




	\section{Introduction} \label{sec:intro}
	
	This work contributes to the study of Algebraic Quantum Symmetry. In particular, we examine when  an algebra $A$ and a Hopf-type algebra $H$ that coacts on $A$ universally\footnote{By $H$ coacting on $A$ universally, we mean that $A$ is an $H$-comodule algebra, so that if $A$ is also a $H'$-comodule algebra, then there is a unique Hopf-type structure map $H \to H'$ compatible with the coactions.} share algebraic properties. This line of research was prompted by Manin's inquiry in the work of Artin-Schelter-Tate \cite{AST} on whether the  bialgebras that coact universally on a well-behaved class of connected graded algebras, Artin-Schelter regular algebras, also enjoy nice ring-theoretic and homological properties. In that work, the question was addressed for key examples of connected graded comodule algebras, skew polynomial rings \cite{AST}, and has been addressed later for other Artin-Schelter regular algebras, especially for the Noetherian, domain, prime, growth conditions and homological dimensions; see, e.g., \cite[Sections~I.2 and II.9]{BG} and \cite{WaltonWang}. But Manin's question is unresolved, in general. Our goal is to address this inquiry for a class of graded algebras that are generally non-connected: path algebras of finite quivers. 
	
	\smallskip
	
	Recall that a \emph{quiver} is simply a directed graph, which is a quadruple $Q=(Q_0,Q_1,s,t)$, where $Q_0$ (resp., $Q_1$) is a  collection of vertices (resp., arrows), and $s,t: Q_1 \rightarrow Q_0$ denote the source and target maps, respectively. We say that $Q$ is \emph{finite} if both $|Q_0|$ and $|Q_1|$ are finite sets, and this will be assumed throughout (see Hypothesis~\ref{hyp:finite}). We  read paths of $Q$ from left-to-right, and all cycles are assumed to be oriented here. Fix an arbitrary field $\kk$. For any quiver $Q$, its associative, unital \emph{path algebra} $\kk Q$ over $\kk$ is the $\kk$-algebra having $\kk$-basis given by the paths of $Q$, with ring structure determined by path concatenation when possible: $a \ast b = \delta_{t(a),s(b)}ab$, for all paths $a, b$ of $Q$. The unit is  $1_{\kk Q} = \sum_{i\in Q_0} e_i$, where each $e_i$ is the trivial path at vertex $i$. The path algebra $\kk Q$ is $\mathbb{N}$-graded by path length, where $(\kk Q)_k= \kk(Q_k)$, for $Q_k$ consisting of paths of  length $k\in \mathbb{N}$.
	
	\smallskip

	Recently, the second author with Huang, Wicks, and Won established a framework for studying universal coactions on non-connected graded algebras via  coactions of {\it weak bialgebras}. Their main result is that the weak bialgebras that coact universally on the path algebra $\kk Q$ (either from the left, from the right, or from both directions compatibly) are each isomorphic to Hayashi's face algebra $\hay$ attached to $Q$ \cite[Theorem~4.17]{HWWW20}. We provide the algebra presentation of $\hay$  below.
	
	\begin{definition}[$\hay$]  \cite[Example~1.1]{Hayashi96}
		\label{ex:hay} 
		For a finite quiver $Q$, the algebra presentation of Hayashi's face algebra $\hay$ attached to $Q$  over a field $\kk$ is given as follows. As a $\kk$-vector space, $\hay$ has a $\kk$-basis of elements $\{x_{a,b}\}$ for $a,b \in Q_k$, for each $k \geq 0$. The ring structure of  $\hay$ is determined by the following relations: 
		\[
		\begin{array}{ll}
			\smallskip
			x_{i,j} x_{i',j'} \; = \; \delta_{i,i'} \delta_{j,j'}  x_{i,j}, &\quad \forall i,j,i',j' \in Q_0,\\
			\smallskip
			x_{s(p),s(q)}x_{p,q} \;= \;x_{p,q} \;= \;x_{p,q} x_{t(p),t(q)}, &\quad \forall p,q \in Q_1,\\
			x_{p,q} x_{p',q'} \;=\; \delta_{t(p), s(p')} \delta_{t(q), s(q')}  x_{pp',qq'}, &\quad
			\forall p,p',q,q' \in Q_1,
		\end{array}
		\]
		with unit given by $1_{\hay} = \textstyle \sum_{i, j \in Q_0} x_{i,j}$. Hayashi's face algebra has a $\mathbb{N}$-grading given by $(\hay)_k= \bigoplus_{a,b\in Q_k} \kk x_{a,b}$, for all $k\in \mathbb{N}$.
	\end{definition}	
	
	It was inquired in \cite[Question 6.5]{HWWW20} whether $\hay$ and $\kk Q$ share nice algebraic properties. We address this question in the result below.
	
	\begin{theorem}[Theorem~\ref{th:mainth}]
		Let $Q$ be a finite quiver,  let $C=(c_{i,j})_{i,j\in Q_0}$  be the adjacency matrix of  $Q$, and let $C^{k}:=(c^{(k)}_{i,j})_{i,j\in Q_0}$ be the adjacency matrix of $Q_k$. Then,
		\begin{enumerate}
		\item[\textnormal{(a)}] The path algebra $\kk Q$ and the face algebra $\hay$ share the following algebraic properties and dimensions: 
		\begin{itemize}
		    \item finite dimensionality; 
		    \item finite Gelfand-Kirillov (GK-)dimensionality;
			 \item (left/right) Noetherianity; 
			\item semiprime; 
			\item global dimension; and
		 \item Koszulity.
		 \end{itemize}
		Moreover, if $\kk Q$ is prime and $Q$ has at least one cycle, then $\hay$ is prime.
		
		\smallskip
		
		\item[\textnormal{(b)}] If $\dim_{\kk} (\kk Q)$ is finite,
		 then
				$\textstyle \dim_{\kk} (\hay) = \sum_{i,j\in Q_0, \; k\geq 0}\; (c_{i,j}^{(k)})^2.$
		
			\smallskip
			
		\item[\textnormal{(c)}] If $\operatorname{GKdim}(\kk Q)$ is finite, then $\operatorname{GKdim}(\hay)=2 \operatorname{GKdim}(\kk Q) - 1 $.
		
		\smallskip
		
		\item[\textnormal{(d)}] The Hilbert series of $\hay$ is given by
			\begin{equation*}
				H_{\hay}(t)=(I \otimes I)+(C\otimes C)t+(C^2\otimes C^2)t^2+\cdots,
			\end{equation*}
		where $I$ is the $|Q_0| \times |Q_0|$ identity matrix, and $\otimes$ is the tensor product of matrices.
		\end{enumerate}
	\end{theorem}
	
	\noindent Indeed, it is well known that $\textstyle \dim_{\kk} (\kk Q) = \sum_{i,j\in Q_0,  k\geq 0}\; c_{i,j}^{(k)}$ and $H_{\kk Q}(t)=(I - Ct)^{-1}$.
	
	The main idea of the proof is to consider the Kronecker square $\widehat{Q}$ of the quiver $Q$ [Definition~\ref{ex:Qhat}], and realize $\hay$ as the path algebra $\kk \widehat{Q}$ [Proposition~\ref{prop:main}] (see also \cite[Remark~3.3]{Pfeiffer}). Then, the theorem above follows from (i) prior results on relating algebraic properties of a path algebra with graph-theoretic properties of the underlying quiver [Proposition~\ref{prop:algprop}] and (ii) a careful analysis of graph-theoretic properties that are shared between a quiver and its Kronecker square [Proposition~\ref{prop:quivprop}]. Many examples of these quivers and their path algebras are provided in Table~\ref{tab:examples}.
	Finally, in Table~\ref{tab:dimADE} we compute $\dim_\kk \kk Q$ and $\dim_\kk \hay$, for some quivers $Q$ of Dynkin type~ADE. 
	
	\medskip
	
	\noindent {\bf Acknowledgements.} 
	We are grateful to Ashish Srivastava for pointing out an error in a previous version of this manuscript on the GK-dimension of $\hay$; it has now been corrected here.	We also thank Marcelo Aguiar, Pablo Ocal, Mohamed Omar, and the anonymous referee for insightful comments and suggestions.

	The first author was supported by the research fund of the Department of Mathematics, Universidad Nacional de Colombia - Sede Bogot\'a, Colombia, HERMES code 52464, and the Fulbright
		Visiting Student Researcher Program. The second author   was partially supported by the US National Science Foundation grants DMS-1903192 and 2100756.


	\section{Preliminaries} \label{sec:prelim}
	
	In this section, we provide an overview of some ring-theoretic and homological properties that will be used in this work; see Section~\ref{sec:properties}. We also recall preliminary graph-theoretic concepts of quivers and corresponding algebraic properties of path algebras; see Section~\ref{sec:quiver}.
	
	\subsection{Ring-theoretic and homological properties} \label{sec:properties}

Take $R$  a ring, $\kk$ an arbitrary field, and $A$ a $\kk$-algebra here.  We refer the reader to \cite[Section 1.1]{GW}, \cite[Sections 0.2, 7.1, 8.1]{MR}, \cite[Section 1]{Fro99}, and \cite[Section 2.1]{EE} for further details of the material here.
	
	\smallskip
	
	A right $R$-module $M_R$ is called \emph{Noetherian} if 
	every submodule of $M$ is finitely generated.
		In particular, a ring $R$ is said to be \emph{right Noetherian} if $R_R$ is Noetherian. Likewise, we can define {\it left Noetherian} rings, and $R$ is {\it Noetherian} if it is both left and right Noetherian.
		
		\smallskip
	
		The \emph{projective dimension} of a module $M_R$, written $\operatorname{pd}(M_R)$, is the shortest length $n$ of a projective resolution of $M$, or equal to 
	   $\infty$ if no such $n$ exists. The \emph{right global dimension} of $R$ is defined by
	$\operatorname{r.gldim}(R):= \sup\{ \operatorname{pd}(M) \mid M \text{ any right $R$-module}\}$.
		Likewise, $\operatorname{l.gldim}(R)$ is defined; when $\operatorname{r.gldim}(R)=\operatorname{l.gldim}(R)$, we simply write $\operatorname{gldim}(R)$. A ring $R$ is called \emph{hereditary} if $\operatorname{gldim}(R)\leq 1$.
	
	 \smallskip
	 
	 An $\N$-graded $\kk$-algebra $A$ is said to be \emph{Koszul} if it has a linear minimal graded free resolution, that is, there exists an exact sequence
	\begin{equation*}
		\cdots \rightarrow A(-i)^{b_i} \rightarrow \cdots \rightarrow A(-2)^{b_2} \rightarrow A(-1)^{b_1} \rightarrow A \rightarrow \kk \rightarrow 0,
	\end{equation*}
	where $A(-j)$ is the graded algebra $A$ with grading shifted up by $j$, that is $A(-j)_i = A_{i-j}$, and the exponents $b_i$ refer to the $b_i$-fold direct sum. 
	
		\smallskip
	
	    A ring $R$ is called \emph{prime} if for every pair of nonzero two-sided ideals $I,J$ of $R$ it follows $IJ\neq 0$, and is called \emph{semiprime} if $R$ has no nonzero nilpotent two-sided ideals. 
		
		\smallskip
	
		Let $A$ be a finitely generated $\kk$-algebra. The \emph{Gelfand-Kirillov dimension} of $A$ is defined by $\operatorname{GKdim}(A)= \sup_V\overline{\lim}_{n\rightarrow \infty} \log_n (\dim_{\kk} V^n )$, where the supremum is taken over all finite dimensional $\kk$-subspaces $V$ of $A$ and $V^n$ denotes the subspace spanned by all elements of the form $v_1\cdots v_n$, where $v_i\in V$, $1\leq i \leq n$. A well known result is that  $\operatorname{GKdim}(A) = 0$ if and only if $A$ is finite-dimensional as a $\kk$-vector space.
	
	\smallskip
	
	Finally, recall that a $\mathbb{N}$-graded bimodule $M=\bigoplus_{n\in \mathbb{N}} M_n$ over a $\mathbb{N}$-graded $\kk$-algebra $A=\bigoplus_{n\in \mathbb{N}} A_n$ is said to be
	\emph{locally finite} if each $M_n$ is finite-dimensional. Let $J$ be a finite set and consider the $\kk$-algebra $A:=\kk^{|J|}$. Note that a locally finite $\mathbb{N}$-graded $A$-bimodule $M$ can be seen as an $J\times J$-graded vector space $M=\bigoplus_{i,j\in J} M_{i,j}$. We define the \emph{(matrix) Hilbert series $h_M(t)$} of  $M$ to be a matrix-valued series with entries given by
		\begin{equation*}
			h_M(t)_{i,j}= \textstyle \sum_{k\geq 0} \dim_\kk ((M_k)_{i,j}) t^k .
		\end{equation*}

	
	\subsection{Properties of quivers and their path algebras} \label{sec:quiver}
	Here, we discuss graph-theoretic properties of quivers, and algebraic properties of their corresponding path algebras. To begin, consider the following assumption.
	
	\begin{hypothesis} \label{hyp:finite}
	We assume throughout this work that quivers are finite.
	\end{hypothesis}
	
	Next, we discuss some graph-theoretic concepts, starting with various types of cycles relevant to our work. Recall that in this paper ``cycle'' means ``oriented cycle''. A quiver is said to be {\it acyclic} if it contains no cycles.
	
	\begin{definition}\label{def:cycle}
		Let $Q=(Q_0,Q_1,s,t)$ be a quiver.
		\begin{enumerate}[label=(\roman*)]
			\item A cycle $p_1 p_2 \cdots p_k\in Q_k$ is called a \emph{simple} if $s(p_i)\neq t(p_j)$ for $2\leq i,j\leq k$.
			\smallskip			
			\item A cycle $c:=p_1 p_2 \cdots p_k $ of $Q$ is said to be a {\it source} (resp., {\it sink}) {\it cycle} if there exists an arrow leaving (resp., entering) $c$, i.e., there exists an arrow $p \in Q_1$, not in $c$, with $s(p) = s(p_i)$ (resp., $t(p) = t(p_i)$) for some $i=1,\dots,k$. 
			
			\smallskip
			
			\item A cycle of $Q$ is called {\it isolated} if it is neither a source nor sink cycle.
			
			\smallskip
			
			\item A cycle of $Q$ is called \emph{exclusive} (or \emph{cyclically simple}) if it is disjoint with every other cycle.

			\smallskip
			
			\item $Q$ is said to satisfy the \emph{exclusive condition} if every cycle of $Q$ is  exclusive.

			\smallskip
			
			\item For two cycles $c,d$ of $Q$, we write $c \Rightarrow d$ if there is a path that starts at a vertex in $c$ and ends at a vertex in $d$. A sequence of distinct cycles $c_1,\ldots,c_n$ of $Q$ is a \emph{chain of cycles} of length $n$ if $c_1 \Rightarrow c_2 \Rightarrow \cdots \Rightarrow c_n$. 

		\end{enumerate}
	\end{definition}
	
	In Figure~\ref{fig:isolated} below, we present examples of (non-)isolated and (non-)exclusive cycles. Notice that every isolated cycle is an exclusive cycle.
	
	\smallskip
	
	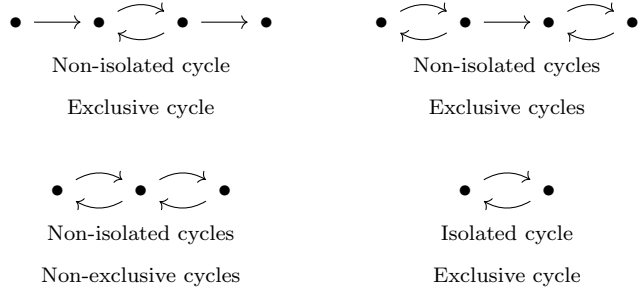
\begin{figure}[!ht]
		\begin{tabular}{c c cc c c}
			$\begin{tikzcd}[sep=scriptsize]
				\bullet \arrow[r] &	\bullet \arrow[r,bend left] & \bullet \arrow[r] \arrow[l,bend left] & \bullet
			\end{tikzcd}$ &&& $\begin{tikzcd}[sep=scriptsize]
				\bullet \arrow[r,bend left] & \bullet \arrow[l,bend left] \arrow[r] & \bullet \arrow[r,bend left] & \bullet \arrow[l,bend left]
			\end{tikzcd}$ \\
			\footnotesize{Non-isolated cycle}&&&\footnotesize{Non-isolated cycles}\\
			\footnotesize{Exclusive cycle} &&& \footnotesize{Exclusive cycles}\\&\\
			$\begin{tikzcd}[sep=scriptsize]
				\bullet \arrow[r,bend left] & \bullet \arrow[r,bend left] \arrow[l,bend left] & \bullet\arrow[l,bend left]
			\end{tikzcd}$ &&& 
			$\begin{tikzcd}[sep=scriptsize]
				\bullet \arrow[r,bend left] & \bullet \arrow[l,bend left]
			\end{tikzcd}$\\
			\footnotesize{Non-isolated cycles} &&& \footnotesize{Isolated cycle}\\
			\footnotesize{Non-exclusive cycles} &&& \footnotesize{Exclusive cycle}
		\end{tabular}
		\caption{Non-/isolated and non-/exclusive cycles}\label{fig:isolated}
	\end{figure}

	Next, we turn our attention to the connected condition of quivers.	
	
	\begin{definition}
		Let $Q=(Q_0,Q_1,s,t)$ be a quiver.
		\begin{enumerate}[label=(\roman*)]
			\item $Q$ is said to be \emph{connected} if for any given decomposition $Q_0=Q_0'\cup Q_0''$, with $Q_0'\cap Q_0''=\emptyset$ and both $Q_0',Q_0''$ non-empty, there exists at least one arrow $p\in Q_1$ such that either $s(p)\in Q_0'$, $t(p)\in Q_0''$ or  $s(p)\in Q_0''$, $t(p)\in Q_0'$.
			
			\smallskip
			
			\item $Q$ is said to be \emph{strongly connected} (or \emph{oriented connected}) if for every $i,j\in Q_0$ with $i\neq j$, there exists a path $p_1\cdots p_k$ such that $s(p_1)=i$ and $t(p_k)=j$.
			
			\smallskip
			\item $Q$ is said to be \emph{pairwise strongly connected} if for every $i,j,i',j'\in Q_0$ with $i\neq j$ and $i'\neq j'$, there exist paths  $p_1\cdots p_k$ and $q_1\cdots q_k$ of the same length such that $s(p_1)=i$, $t(p_k)=j$, $s(q_1)=i'$ and $t(q_k)=j'$.
			
			\smallskip
			
			\item $Q$ is said to be \emph{path reversible} if for every path $p_1p_2\cdots p_k\in Q_k$, there exists a path $q_1 q_2 \cdots q_l\in Q_l$ such that $s(p_1)=t(q_l)$ and $t(p_k)=s(q_1)$. Here, $l$ need not equal $k$.
		\end{enumerate}
	\end{definition}

	It is clear that pairwise strongly connected $\Rightarrow$ strongly connected $\Rightarrow$ connected. However, the converses do not hold as we see in Figure~\ref{fig:connected} below.
	
	\begin{figure}[!ht]
		\begin{tabular}{c cc c}
			$\begin{tikzcd}[sep=scriptsize]
				\bullet \arrow[r] & \bullet
			\end{tikzcd}$ &&& $\begin{tikzcd}[sep=scriptsize]
				& \bullet \arrow[rd] \\
				\bullet \arrow[ru] && \bullet \arrow[ll]
			\end{tikzcd}$\\			
			\footnotesize{Connected} &&&\footnotesize{Strongly connected}\\
			\footnotesize{Not strongly connected} &&& \footnotesize{Not pairwise strongly connected}
		\end{tabular}
		\caption{Various connected quivers}\label{fig:connected}
	\end{figure}
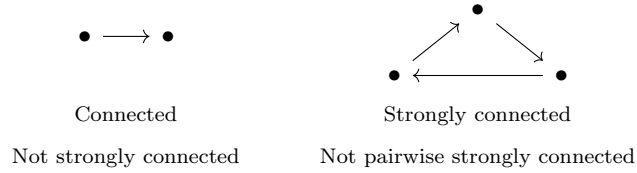

	Now we recall results on algebraic properties of path algebras which depend on graph-theoretic properties of the underlying quiver; many of these results are standard.
	
	\begin{notation}[$C$, $C^k$, $c_{i,j}^{(k)}$] \label{not:adj}
		For a quiver $Q$, denote by $C=(c_{i,j})_{i,j\in Q_0}$ the adjacency matrix of (arrows in) $Q$, and by $C^{k}:=(c^{(k)}_{i,j})_{i,j\in Q_0}$  the adjacency matrix of paths of length $k$ in $Q$ (which is equal to the $k$-th power of $C$).
	\end{notation}
	
	\begin{proposition} \label{prop:algprop}
		Let $Q=(Q_0,Q_1,s,t)$ be a quiver.
		\begin{enumerate}[label=\normalfont(\roman*)]
			\item[\textnormal{(i)}] 
			\begin{enumerate}
				\item[\textnormal{(a)}] \cite[Lemma~II.1.4]{ASS2006}
				$\kk Q$ is finite dimensional if and only if $Q$ is acyclic. 
				
				\smallskip
				
				\item[\textnormal{(b)}] 
				When $\kk Q$ is finite dimensional, $\dim_{\kk} \kk Q = \sum_{i,j\in Q_0,\; k\geq 0} c_{i,j}^{(k)}.$
				\smallskip
				
				\item[\textnormal{(c)}] In general, the Hilbert series of $\kk Q$ is given by
				\begin{equation*}
					H_{\kk Q}(t)=(I-Ct)^{-1}=I+Ct+C^{2}t^2+C^{3}t^3+\cdots,
				\end{equation*}
				where $I$ denotes the $|Q_0| \times |Q_0|$ identity matrix. 
			\end{enumerate}
			
			\smallskip
			
			\item[\textnormal{(ii)}]  \cite{Ufnar} \textnormal{(see also \cite[Theorem~3.12]{MorenoMolina})} $\kk Q$ has finite GK-dimension if and only if $Q$ satisfies the exclusive condition. In this case, $\operatorname{GKdim}(\kk Q)$ equals  the maximal length of chains of cycles in $Q$. 
			
			\smallskip
			
			\item[\textnormal{(iii)}] \textnormal{(see, e.g., \cite[Theorems~2.2 and~2.3]{CL})} $\kk Q$ is (resp., right, left) Noetherian if and only if every cycle in $Q$ is (resp., not a source cycle, not a sink cycle) isolated.
			
			\smallskip
			
			\item[\textnormal{(iv)}]  \textnormal{(see, e.g., \cite[Theorem~2.1]{CL})}  $\kk Q$ is prime if and only if $Q$ is strongly connected.
			
			\smallskip
			
			\item[\textnormal{(v)}]  \cite[Proposition~2.1]{SilesMolina} $\kk Q$ is semiprime if and only if $Q$ is path reversible. 
			
			\smallskip
			
			\item[\textnormal{(vi)}] \textnormal{(see, e.g., \cite[Section~1.4]{Brion})} $\kk Q$ is hereditary,  and  $\operatorname{gldim}(\kk Q)= 0$ if and only if $Q$ is arrowless.  \qed
		\end{enumerate}
	\end{proposition}
	
	\section{Main results}
	
	In this section, we present our main result on  algebraic properties of face algebras $\hay$; see Section~\ref{sec:hay}. This is done by using the Kronecker square of $Q$, examined in Section~\ref{sec:Kronecker}.

	\subsection{On the Kronecker square of a quiver and its path algebra} \label{sec:Kronecker} 
	
	The construction of the quiver  below plays a key role in this work.
	
	\begin{definition}[$\widehat{Q}$]\label{ex:Qhat}
		Let $Q=(Q_0,Q_1,s,t)$ be a quiver. We define the {\it Kronecker square}  $\widehat{Q}$ of $Q$ as the quiver $\widehat{Q}=(\widehat{Q}_0, \widehat{Q}_1, \widehat{s}, \widehat{t})$ given by
		\[
		\begin{array}{lll}
			\medskip
			&\widehat{Q}_0= \{ [i,j] \}_{i,j\in Q_0},  \qquad & \widehat{Q}_1 = \{ [p,q]\}_{p,q\in Q_1}, \\
			& \widehat{s}([p,q])=[s(p),s(q)],  & \widehat{t}([p,q])=[t(p),t(q)],\quad \text{for all $p,q \in Q_1$}.
		\end{array}
		\]
	\end{definition} 
	
	Any path in $\widehat{Q}$ is of the form $[p_1 \cdots p_k, q_1 \cdots q_k]:=[p_1,q_1][p_2,q_2]\cdots[p_k,q_k]$ where $p_1\cdots p_k$, $q_1\cdots q_k \in Q_k$.
	By construction, $|\widehat{Q}_0|=|Q_0|^2$ and $|\widehat{Q}_1|=|Q_1|^2$. Moreover, $Q$ can be identified with the subquiver of $\widehat{Q}$ formed with the vertices $\{[i,i]\}_{i\in Q_0}$ and arrows $\{[p,p]\}_{p\in Q_1}$. Therefore, we have an embedding of path algebras $\kk Q \hookrightarrow \kk\widehat{Q}$.
	The reader may wish to refer to Table~\ref{tab:examples} below for examples of quivers $Q$, their Kronecker square $\widehat{Q}$, and their corresponding path algebras.
	
	\smallskip
	
	Next, we compare graph-theoretic properties of a quiver $Q$ with that of its Kronecker square $\widehat{Q}$, which will be used in the proof of Theorem~\ref{th:mainth} below. Recall Notation~\ref{not:adj}.
	
	\begin{proposition}\label{prop:quivprop}
		Let $Q=(Q_0,Q_1,s,t)$ be a quiver. Then, the following statements hold.
		\begin{enumerate}[label=\normalfont(\roman*)]
			\item \begin{enumerate}
				\item[\textnormal{(a)}] $Q$ is finite (resp., acyclic) if and only if $\widehat{Q}$ is finite (resp., acyclic). 
				
				\smallskip
				
				\item[\textnormal{(b)}] $|\widehat{Q}_k|=|Q_k|^2$, $k\geq 0$.
				
				\smallskip
				
				\item[\textnormal{(c)}] The adjacency matrix $\widehat{C}^k$ of  $\widehat{Q}_k$ is given by $\widehat{C}^k=C^k\otimes C^k$, for $k \geq 0$, where $\otimes$ is the tensor product of matrices.
			\end{enumerate}
			
			\smallskip
			
			\item
			$Q$ satisfies the exclusive condition if and only if $\widehat{Q}$ satisfies the exclusive condition.
			In this case, if the maximal length of chains of cycles in $Q$ is $n$, then the maximal length of chains of cycles in $\widehat{Q}$ is $2n-1$.

			\smallskip
			
			\item $Q$ has a source (resp., \c{sink}) cycle if and only if $\widehat{Q}$ has a source (resp., sink) cycle. 
			
			\smallskip
			
			\item 
			\begin{enumerate}
				\item[\textnormal{(a)}] $Q$ is pairwise strongly connected if and only if $\widehat{Q}$ is strongly connected.
				
				\smallskip
				
				\item[\textnormal{(b)}] If $Q$ is strongly connected and has at least one cycle, then $\widehat{Q}$ is strongly connected. Conversely, if $\widehat{Q}$ is strongly connected, then $Q$ is strongly connected.
			\end{enumerate}
			
			\smallskip
			
			\item $Q$ is path reversible if and only if $\widehat{Q}$ is path reversible. 
			
			\smallskip
			
			\item $Q$ is arrowless if and only if $\widehat{Q}$ is arrowless.
			
			\smallskip

		\end{enumerate}
	\end{proposition}
	
	\begin{proof}
		(i.a) The statement about the finite condition follows from the definition of $\widehat{Q}$.
		
		If $\widehat{Q}$ is acyclic, then by identifying $Q$ as a subquiver of $\widehat{Q}$, we see that $Q$ is also acyclic. On the other hand, if  there is a cycle $[p_1p_2 \cdots p_k,q_1 q_2 \cdots q_k]\in \widehat{Q}_k$, then
		\begin{equation*}
			[s(p_1),s(q_1)]=\widehat{s}([p_1,q_1])=\widehat{t}([p_k,q_k])=[t(p_k),t(q_k)].
		\end{equation*}
		This implies that $Q$ contains the cycles $p_1\cdots p_k$ and $q_1\cdots q_k$.
		
		\smallskip 
		
		(i.b) Any path $[p_1 \cdots p_k, q_1 \cdots q_k]$ of length $k$ in $\widehat{Q}$ uniquely corresponds to the pair of paths $p_1 \cdots p_k$ and $q_1 \cdots q_k$ in $Q$. This means that the sets $\widehat{Q}_k$ and $Q_k\times Q_k$ are in one-to-one correspondence. Thus, $|\widehat{Q}_k|=|Q_k|^2$. 
		
		\smallskip	
		
		(i.c) Consider the adjacency matrix of paths of length $k$ in $\widehat{Q}$: $$\widehat{C}^k=\left(d_{[i,i'],[j,j']}^{(k)}\right)_{i,i',j,j'\in Q_0}.$$ On the other hand, each entry of $C^k\otimes C^k$ is of the form $c_{i,j}^{(k)}c_{i',j'}^{(k)}$, with $ i,i',j,j'\in Q_0$. Suppose that $c_{i,j}^{(k)}=n \geq 0$ and $c_{i',j'}^{(k)}=m \geq 0$, that is, there are $n$ paths of length $k$ from $i$ to $j$, and $m$ paths of length $k$ from $i'$ to $j'$. This determines $nm$ paths of length $k$ from $[i,i']$ to $[j,j']$ in $\widehat{Q}$, and no other of such paths can be formed. So, $d_{[i,i'],[j,j']}^{(k)}=nm$. Hence, $\widehat{C}^k=C^k\otimes C^k$.
		
		\smallskip

		(ii) Suppose that there exists a non-exclusive cycle $\widehat{c}=[p_1 \cdots p_k, q_1 \cdots q_k]$ in $\widehat{Q}$. Then, there is another cycle $\widehat{c'}=[p'_1 \cdots p'_l, q'_1 \cdots q'_l]$ in $\widehat{Q}$ such that $[s(p_i),s(q_i)]=[s(p'_j),s(q'_j)]$ for some $i=1,\dots,k$ and $j = 1,\dots,l$. This yields cycles $c_1=p_1\cdots p_k$ and $c_2=q_1 \cdots q_k$ in $Q$ that are non-exclusive: $c_1$ is not disjoint with $c'_1=p'_1 \cdots p'_l$ for $s(p_i)=s(p'_j)$, and $c_2$ is not disjoint with $c'_2=q'_1 \cdots q'_l$, for $s(q_i)=s(q'_j)$. Therefore, $Q$ also fails the exclusive condition.
		
		Conversely, suppose that $Q$ fails the exclusive condition. Then, using the fact that $Q$ can be identified with a sub-quiver of $\widehat{Q}$, we see that if two cycles in $Q$ were not disjoint, they would be also non-disjoint in $\widehat{Q}$. Thus, $\widehat{Q}$ fails the exclusive condition.
		
		\smallskip
		
		For the second part, denote the maximum length of a chain of cycles in $Q$ by $\operatorname{cyc.len}(Q)$. We aim to show that $\operatorname{cyc.len}(\widehat{Q}) = 2 \operatorname{cyc.len}(Q)- 1$. This is done with the notes below.
		
		\begin{itemize}
		    \item[1.] Given two distinct cycles $c=p_1p_2 \cdots c_k$, $d=q_1q_2 \cdots q_l$ in $Q$, construct the cycle $\widehat{c,d}$ in $\widehat{Q}$ given by 
		    \[	\widehat{s}([c^u,d^v])=[s(p_1),s(q_1)]=[t(p_k),t(q_l)]=\widehat{t}([c^u,d^v]),\]
		    for  $u:=\operatorname{lcm}(k,l)/k$ and $v:=\operatorname{lcm}(k,l)/l$. Indeed,
		\begin{gather*}
			\operatorname{length}(c^u) = u \operatorname{length}(c)=uk=vl=v \operatorname{length}(d)=\operatorname{length}(d^v).
		\end{gather*}
	
		     \item[2.] Note that all simple cycles in $\widehat{Q}$ are of the form $\widehat{c,d}$ for some (not necessarily distinct) cycles $c,d$ in $Q$.
		     
		     \smallskip
		     
		     \item[3.] Note that if $c_1,c_2,d_1,d_2$ are cycles in $Q$ such that $c_1 \Rightarrow d_1$ and $c_2 \Rightarrow d_2$, we have that $\widehat{c_1,d_1} \Rightarrow \widehat{c_2,d_2}$ in $\widehat{Q}$.
		     
		     \smallskip
		     
		     \item[4.] If $\widehat{c_1,d_1},\widehat{c_2,d_2}$ are simple cycles in $\widehat{Q}$ such that $\widehat{c_1,d_1} \Rightarrow \widehat{c_2,d_2}$, then $c_1 \Rightarrow c_2$ and $d_1 \Rightarrow d_2$ in $Q$.
		     
		     \smallskip
		     
	         \item[5.] Any chain of cycles of length $n$ in $Q$ induces a chain of cycles of length $2n-1$ in $\widehat{Q}$. Indeed, let $c_1 \Rightarrow c_2 \Rightarrow \cdots \Rightarrow c_n$ be a chain of cycles of length $n$ in $Q$. Then the chain of cycles
\begin{equation*}
\widehat{c_1,c_1} \Rightarrow \widehat{c_1,c_2} \Rightarrow \cdots \Rightarrow \widehat{c_1,c_n} \Rightarrow \widehat{c_2,c_n} \Rightarrow \widehat{c_3,c_n} \Rightarrow \cdots \Rightarrow \widehat{c_n,c_n}
\end{equation*}
in $\widehat{Q}$ has length $n+(n-1)=2n-1$. The arrows were constructed as in Note 3.
\end{itemize}
		
To prove the claim, observe that by the result already proven in this part, $\widehat{Q}$ satisfies the exclusive condition, and hence all cycles $\widehat{Q}$ are either simple (of the form $\widehat{c,d}$ by Notes 1 and 2), or powers of simple cycles. Therefore, in any chain of cycles in $\widehat{Q}$, we can assume that all cycles are simple. Let $n:=\operatorname{cyc.len}(Q)$, and suppose that there is a chain of cycles $\widehat{c_1,d_1} \Rightarrow \widehat{c_2,d_2} \Rightarrow  \cdots \Rightarrow  \widehat{c_m,d_m}$ in $\widehat{Q}$. Then by Note 4, we have $c_1 \Rightarrow c_2, \;c_2 \Rightarrow c_3,\; \ldots, c_{m-1}\Rightarrow c_m$ in $Q$. Observe that  we only have at most $n$ distinct choices for the $c_i$. A similar argument can be done for the $d_i$ to deduce that there are  at most $n$ distinct choices.
 So when constructing the cycles, $\widehat{c_i,d_i}$, at most  $n+(n-1)=2n-1$ distinct pairs are possible, that is, $m \leq 2n-1$. Hence, $\operatorname{cyc.len}(\widehat{Q})\leq 2  \operatorname{cyc.len}(Q) -1$. But in Note 5 we used a chain of cycles of length $n$ in $Q$ to construct a chain of cycles of length exactly $2n-1$ in $\widehat{Q}$. So, $\operatorname{cyc.len}(\widehat{Q}) = 2  \operatorname{cyc.len}(Q) -1$, as desired.

\smallskip
		
		(iii) Suppose that $Q$ has a source cycle $c:=p_1 p_2 \cdots p_k $ with arrow $p \in Q_1$, not in $c$, such that $s(p) = s(p_i)$ for some $i=1,\dots,k$. Then $\widehat{c}=[p_1\cdots p_k,p_1\cdots p_k]$ will be a source cycle in $\widehat{Q}$, where $[p,p]$ is an arrow leaving $\widehat{c}$ at $[p_i,p_i]$.
		
		Conversely, suppose that $\widehat{Q}$ has a source cycle, that is, a cycle $\widehat{c}=[p_1 \cdots p_k, q_1 \cdots q_k]$ with arrow $[p,q] \in \widehat{Q}_1$, not in $\widehat{c}$, such that $[s(p),s(q)]=[s(p_i),s(q_i)]$ for some $i=1,\dots,k$. This intermediately induces two source cycles in $Q$: $c_1=p_1 \cdots p_k$ with arrow $p$ leaving at $p_i$ and $c_2=q_1 \cdots q_k$ with arrow $q$ leaving at $q_i$.
		
		The argument for sink cycles is similar.
		
		\smallskip		
		
		(iv.a) Suppose that $Q$ is pairwise strongly connected. Then for every pair of vertices $[i,i'],[j,j']\in \widehat{Q}_0$ with $[i,i']\neq [j,j']$, there exist paths of the same length $p_1 \cdots p_k$ and $q_1 \cdots q_k$ such that $s(p_1)=i$, $t(p_k)=j$, $s(q_1)=i'$ and $t(q_k)=j'$. Thus, we have the path $[p_1 \cdots p_k,q_1 \cdots q_k] \in \widehat{Q}_k$ connecting $[i,i']$ and $[j,j']$. Hence $\widehat{Q}$ is strongly connected.
		
		Conversely, if $i,i',j,j' \in Q_0$ are such that  $[i,i']\neq [j,j']$ in $\widehat{Q}_0$, then  $\widehat{Q}$ being strongly connected implies that there exists a path $[p_1 \cdots p_k,q_1 \cdots q_k] \in \widehat{Q}_k $ with $\widehat{s}[p_1,q_1]=[i,i']$ and $\widehat{t}[p_k,q_k]=[j,j']$. Therefore, the paths $p_1 \cdots p_k,q_1 \cdots q_k \in Q_k$ are such that $s(p_1)=i$, $t(p_k)=j$, $s(q_1)=i'$, and $t(q_k)=j'$. Hence, $Q$ is pairwise strongly connected.
		
		\smallskip		
		
		(iv.b) Suppose that $Q$ is strongly connected and has a cycle $c$. Take $i,i',j,j' \in Q_0$ such that $i\neq j$ and $i'\neq j'$. If $x\in Q_0$ is a vertex of the cycle $c$ in $Q$, then take the following paths in $Q$: $p_1 \cdots p_n$ connecting $i$ with $x$, and $\overline{p}_1 \cdots \overline{p}_m$  connecting $x$ with $j$, $q_1 \cdots q_u$ connecting $i'$ with $x$, and $\overline{q}_1 \cdots \overline{q}_v$  connecting $x$ with $j'$. Then
		\[
		p_1 \cdots p_n c^{u+v} \overline{p}_1 \cdots \overline{p}_m, \qquad q_1 \cdots q_u c^{n+m} \overline{q}_1 \cdots \overline{q}_v \in Q_{n+m+u+v}
		\]
		\c{are paths (of the same length)} that connect $i$ with $j$, and $i'$ with $j'$, respectively. Then, $Q$ is pairwise strongly connected, and thus $\widehat{Q}$ is strongly connected by (iv.a).
		
		Conversely, suppose that $\widehat{Q}$ is strongly connected. If $i,j\in Q_0$ are such that $i\neq j$, then $[i,i]\neq[j,j]$ in $\widehat{Q}_0$. By the hypothesis, there exists a path $[p_1 \cdots p_k,q_1 \cdots q_k]$ connecting $[i,i]$ with $[j,j]$ in $\widehat{Q}$. Hence, $p_1 \cdots p_k$ and $q_1 \cdots q_k$ are both paths connecting $i$ with $j$ in $Q$. That is, $Q$ is strongly connected.
		
		\smallskip		
		
		(v) Suppose that $\widehat{Q}$ is path reversible, and take a path $p_1 \cdots p_k$ in $Q$. Consider the corresponding path $[p_1 \cdots p_k, p_1 \cdots p_k]$ in $\widehat{Q}$. Then there exists a path $[q_1\cdots q_l, q'_1 \cdots q'_l]$ in  $\widehat{Q}$ such that $[s(p_1),s(p_1)]=\widehat{s}[p_1,p_1]=\widehat{t}[q_l,q'_l]=[t(q_l),t(q'_l)]$ and $[s(q_1),s(q_1)]=\widehat{s}[q_1,q_1]=\widehat{t}[p_k,p_k]=[t(p_k),t(p_k)]$. That is, the paths $q_1 \cdots q_l, q'_1 \cdots q'_l \in Q_l $ are both reverse paths for $p_1 \cdots p_k$, and $Q$ is path reversible. 
		
		Towards the converse, we prove the following statement: if $Q$ is path reversible, then given any two paths of the same length, their respective reverse paths can be taken also of the same length. Indeed, if $a,b\in Q_k$ are two paths of $Q$, there are reverse paths $a'\in Q_u$ and $b'\in Q_v$ resp., with $u$ not necessarily is equal to $v$. If we take
		\[
		a'':=a'(aa')^{k+v-1} \quad \text{and} \quad b'':=b'(bb')^{k+u-1},
		\]
		notice that $s(a'')=s(a')=t(a)$, $t(a'')=t(a')=s(a)$, $s(b'')=s(b')=t(b)$ and $t(b'')=t(b')=s(b)$. Moreover,
		$$\operatorname{length}(a'')=u+(k+v-1)(k+u)=v+(k+u-1)(k+v) = \operatorname{length}(b'').$$
		Hence, $a'',b''$ are reverse paths of $a,b$, resp., of the same length.
		
		Now take $[p_1 \cdots p_k,q_1 \cdots q_k]\in \widehat{Q}_k$ and suppose that $Q$ is path reversible. Since $p_1 \cdots p_k$ and $q_1 \cdots q_k \in Q_k$ are paths of $Q$ of the same length, by the above, there exist reverse path of the same length $p'_1 \cdots p'_l,q'_1 \cdots q'_l \in Q_l$. That is, $[p'_1 \cdots p'_l,q'_1 \cdots q'_l]$ is the reverse path of $[p_1 \cdots p_k,q_1 \cdots q_k]$ in $\widehat{Q}$.
		
		\smallskip	
		
		(vi) This follows from the definition of $\widehat{Q}$.
	\end{proof}
	
	\begin{remark}	\label{rem:no-loop}
		Regarding Proposition~\ref{prop:quivprop}(iv.b),  being strongly connected does not necessarily imply that $\widehat{Q}$ is strongly connected, as shown in Figure~\ref{fig:strongly} below.
	\end{remark}
	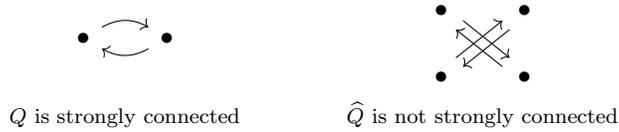
\begin{figure}[!ht]
		\begin{tabular}{ccccc}
			$\begin{tikzcd}[sep=scriptsize]
				\bullet \arrow[r,bend left] & \bullet \arrow[l,bend left]
			\end{tikzcd}$ &&&& $\begin{tikzcd}[sep=scriptsize]
				\bullet \arrow[dr,shift left=.5ex] & \bullet \arrow[dl,shift left=.5ex]\\
				\bullet \arrow[ur,shift left=.5ex] & \bullet \arrow[lu,shift left=.5ex]
			\end{tikzcd}$\\
			\footnotesize{$Q$ is strongly connected} &&&& \footnotesize{$\widehat{Q}$ is not strongly connected}
		\end{tabular}
		\caption{On the strongly connected condition} \label{fig:strongly}
	\end{figure}

	
	\subsection{On Hayashi's face algebra $\hay$} \label{sec:hay}
	We provide the main result on algebraic properties of Hayashi's face algebras $\hay$ [Definition~\ref{ex:hay}] in this section. Before this, we need the following preliminary result, which was also mentioned in \cite[Remark~3.3]{Pfeiffer}.
	
	\begin{proposition} \label{prop:main}
		Let $Q$ be a  quiver, and consider its Kronecker square $\widehat{Q}$ \textnormal{[Definition~\ref{ex:Qhat}]}. Then, $\hay \cong \kk \widehat{Q}$ as unital $\mathbb{N}$-graded $\kk$-algebras. 
	\end{proposition}
	
	\begin{proof}
		First, notice that the product of paths in $\kk \widehat{Q}$ is defined in terms of the concatenation of paths in $Q$, that is, $[a,b][c,d]=\delta_{t(a),s(c)}\delta_{t(b),s(d)}[ac,bd]$ for any paths $a,b$, and paths $c,d$, of the same length. The unit of $\kk \widehat{Q}$ is given by $1_{\kk \widehat{Q}} = \sum_{i,j\in Q_0} e_{[i,j]},$
		where $e_{[i,j]}$ denotes the trivial path at vertex $[i,j]\in \widehat{Q}_0$.
		Now consider the map $\varphi : \hay \rightarrow \kk \widehat{Q}$ given by 
		$$
		\varphi(x_{a,b})=[a,b], \quad a,b\in Q_k,\; k\geq 1 \qquad \text{and} \qquad
		\varphi(x_{i,j})=e_{[i,j]}, \quad i,j\in Q_0.
		$$
		Since $\hay = \bigoplus_{k \geq 0; a,b \in Q_k} \kk x_{a,b}$ as $\kk$-vector spaces,  $\varphi$ is a $\kk$-linear map. Moreover,
		\begin{gather*}
			\varphi(x_{a,b}x_{c,d})=\delta_{t(a),s(c)}\delta_{t(b),s(d)}\varphi(x_{ac,bd})=\delta_{t(a),s(c)}\delta_{t(b),s(d)}[ac,bd]=\varphi(x_{a,b})\varphi(x_{c,d}),\\
			\varphi(1_{\hay})=\varphi\left(\textstyle \sum_{i,j\in Q_0} x_{i,j} \right)= \textstyle \sum_{i,j\in Q_0} \varphi(x_{i,j})=\textstyle \sum_{i, j \in Q_0} e_{[i,j]}=1_{\kk \widehat{Q}}.
		\end{gather*}
		Therefore, $\varphi$ is a unital $\kk$-algebra map, which is clearly bijective and graded. 
	\end{proof}
	
	This brings us to the main result on algebraic properties of the face algebra $\hay$.
	
	\begin{theorem}
		\label{th:mainth}
		Let $Q$ be a  quiver, and recall Notation~\ref{not:adj}. Then, the following hold. 
		\begin{enumerate}[label=\normalfont(\roman*)]
			\item $\kk Q$ is finite dimensional if and only if $\hay$ is finite dimensional. In this case, 
			\begin{equation*}
				\textstyle	\dim_{\kk} (\hay) = \sum_{i,j\in Q_0, \; k\geq 0}\; (c_{i,j}^{(k)})^2.
			\end{equation*}
			\smallskip
			In general, the Hilbert series of $\hay$ is given by
			\begin{equation*}
				H_{\hay}(t)=(I \otimes I)+(C\otimes C)t+(C^2\otimes C^2)t^2+\cdots,
			\end{equation*}
			where $I$ is the $|Q_0| \times |Q_0|$ identity matrix, and $\otimes$ is the tensor product of matrices.
			\smallskip
			\item $\kk Q$ has finite GK-dimension if and only if $\hay$ has finite GK-dimension. In this case,
		    $\operatorname{GKdim}(\hay)= 2\operatorname{GKdim}(\kk Q)-1.$
			\smallskip
			\item $\kk Q$ is (left, right) Noetherian if and only if $\hay$ is (left, right) Noetherian.
			\smallskip
			\item If $\kk Q$ is prime and $Q$ has at least one cycle, then $\hay$ is prime. Conversely, if $\hay$ is prime, then $\kk Q$ is prime.
			\smallskip
			\item $\kk Q$ is semiprime if and only if $\hay$ is semiprime.
			\smallskip
			\item $\operatorname{gldim}(\kk Q)=\operatorname{gldim}(\hay)$; in particular, $\hay$ is hereditary.
			\smallskip
			\item $\kk Q$ and $\hay$ are Koszul.
		\end{enumerate}
	\end{theorem}
	
	\begin{proof}
		By Proposition~\ref{prop:main}, it suffices to establish the statements above with  $\kk \widehat{Q}$ in place of  $\hay$. Now the statements (i)-(vi) follow from parts (i)-(vi) of Propositions~\ref{prop:algprop} and~\ref{prop:quivprop}, respectively. Part (vii) holds as the path algebra $\kk Q$ (resp., $\kk \widehat{Q} \cong \hay)$ is realized as a tensor algebra $T_{\kk Q_0}(\kk Q_1)$ (resp., $T_{\kk \widehat{Q}_0}(\kk \widehat{Q}_1)$, and tensor algebras are Koszul. 
	\end{proof}
	
	\begin{remark}
		For $Q$ acyclic, $\kk Q$ being prime may not imply $\hay$ is prime [Remark~\ref{rem:no-loop}].
	\end{remark}
	

	\section{Examples} \label{sec:example}
	
	In Table~\ref{tab:examples} below, we present some quivers $Q$ and $\widehat{Q}$ and their corresponding path algebras illustrating the results obtained in the previous sections. 
	Moreover, in Table~\ref{tab:dimADE} below, we present a summary of the results for dimension of path algebras of quivers having as underlying graph a Dynkin diagram; note that orientation of the quiver is relevant. 
	
	\clearpage
	
	\textcolor{white}{.}

\vspace{.3in}

	{\footnotesize	
		\begin{table}[!ht]
			\begin{tabular}{|c|c|c|c|}
				\hline 
				$Q$ & $\kk Q$ & $\widehat{Q}$ & $\kk \widehat{Q} \cong \hay$\\\hline\hline
				$\begin{matrix}
					\bullet_1 & \bullet_2 & \cdots & \bullet_n
				\end{matrix}$ & $\kk^n$ & $\begin{matrix}
					\bullet_{[1,1]} & \cdots & \bullet_{[1,n]}\\
					\vdots & \ddots & \vdots\\
					\bullet_{[n,1]} & \cdots & \bullet_{[n,n]}
				\end{matrix}$ & $\kk^{n^2}$\\\hline
				$\begin{tikzcd}[sep=small]
					\bullet \arrow[out=0,in=30,loop,swap,"p_1"]
					\arrow[out=90,in=120,loop,swap,"p_2"]
					\arrow[out=180,in=210,loop,swap,"\cdots"]
					\arrow[out=270,in=300,loop,swap,"p_n"]
				\end{tikzcd}$ & $\kk \langle t_1,\ldots,t_n \rangle $ & $\begin{tikzcd}[sep=small]
					\bullet \arrow[out=0,in=36,loop,swap,"{[p_1,p_1]}"]
					\arrow[out=72,in=108,loop,swap,"\cdots"]
					\arrow[out=144,in=180,loop,swap,"{[p_i,p_j]}"]
					\arrow[out=216,in=256,loop,swap,"\cdots"]
					\arrow[out=288,in=324,loop,swap,"{[p_n,p_n]}"]
				\end{tikzcd}$ & $\kk \langle t_{i,j} \rangle_{i,j=1}^n$\\\hline
				$\begin{tikzcd}[sep=scriptsize]
					\bullet \arrow[r] & \bullet
				\end{tikzcd}$ & $T_2(\kk) = \begin{bmatrix}
					\kk & \kk\\
					0 & \kk
				\end{bmatrix}$ & $\begin{tikzcd}[sep=scriptsize]
					\bullet \arrow[rd] & \bullet\\
					\bullet & \bullet
				\end{tikzcd}$ & $T_2(\kk) \times \kk^2$\\\hline
				$\begin{tikzcd}[sep=scriptsize]
					\bullet \arrow[r] & \bullet \arrow[r] & \bullet
				\end{tikzcd}$ & $T_3(\kk) = \begin{bmatrix}
					\kk & \kk & \kk\\
					0 & \kk & \kk\\
					0 & 0 & \kk
				\end{bmatrix}$ & $\begin{tikzcd}[sep=scriptsize]
					\bullet\arrow[rd] & \bullet \arrow[rd] & \bullet\\
					\bullet\arrow[rd] & \bullet\arrow[rd] & \bullet\\
					\bullet & \bullet & \bullet
				\end{tikzcd}$ & $T_3(\kk) \times T_2(\kk)^2 \times \kk^2$\\\hline
				$\begin{tikzcd}[sep=scriptsize]
					\bullet \arrow[r] & \bullet & \bullet \arrow[l]
				\end{tikzcd}$ & $\begin{bmatrix}
					\kk & 0 & 0\\
					\kk & \kk & 0\\
					\kk & 0 & \kk
				\end{bmatrix}$ & $\begin{tikzcd}[sep=scriptsize]
					\bullet\arrow[rd] & \bullet  & \arrow[ld]\bullet\\
					\bullet & \bullet & \bullet \\
					\bullet \arrow[ru] & \bullet & \bullet\arrow[ul]
				\end{tikzcd}$ & $\begin{bmatrix}
					\kk & 0 & 0 & 0 &0\\
					\kk & \kk & 0 &0 &0\\
					\kk & 0 & \kk & 0 & 0\\
					\kk & 0 & 0 & \kk & 0\\
					\kk & 0 & 0 & 0 & \kk
				\end{bmatrix} \times \kk^4$\\\hline
					$\begin{tikzcd}[sep=small]
					\bullet \arrow[r] & \bullet & \bullet \arrow[r] & \bullet
				\end{tikzcd}$ & $T_2(\kk)^2$ & $\begin{tikzcd}[sep=small]
					\bullet \arrow[rd] & \bullet & \bullet  \arrow[rd] & \bullet\\
					\bullet & \bullet & \bullet & \bullet\\
					\bullet  \arrow[rd] & \bullet & \bullet  \arrow[rd] & \bullet\\
					\bullet & \bullet & \bullet & \bullet
				\end{tikzcd}$ & $T_2(\kk)^4 \times \kk^8$\\\hline
				$\begin{tikzcd}[sep=scriptsize]
					\bullet \arrow[r,shift left=1ex]\arrow[r,shift right=1ex] & \bullet
				\end{tikzcd}$ & $ \begin{bmatrix}
					\kk & \kk^2\\
					0 & \kk
				\end{bmatrix}$ & $\begin{tikzcd}[sep=scriptsize]
					\bullet\arrow[rd,shift left=1.5ex] \arrow[rd,shift left=0.5ex] \arrow[rd,shift right=1.5ex] \arrow[rd,shift right=0.5ex] & \bullet \\
					\bullet & \bullet
				\end{tikzcd}$ & $\begin{bmatrix}
					\kk & \kk^4\\
					0 & \kk
				\end{bmatrix} \times \kk^2$\\\hline
				$\begin{tikzcd}[sep=scriptsize]
					\bullet \arrow[r] \arrow[out=60,in=120,loop] & \bullet
				\end{tikzcd}$ & $ \begin{bmatrix}
					\kk[t] & \kk[t]\\
					0 & \kk
				\end{bmatrix}$ & $\begin{tikzcd}[sep=scriptsize]
					\bullet\arrow[out=60,in=120,loop] \arrow[r] \arrow[rd] \arrow[d] & \bullet \\
					\bullet & \bullet
				\end{tikzcd}$ & $\begin{bmatrix}
					\kk [t] & \kk[t] & \kk[t] & \kk[t]\\
					0 & \kk & 0 & 0\\
					0 & 0 & \kk & 0\\
					0 & 0 & 0 & \kk
				\end{bmatrix}$\\\hline
			\end{tabular}
			\vspace{.2in}
			
			\caption{Examples of $Q$, $\widehat{Q}$, and their path algebras}\label{tab:examples}
		\end{table}
	}
	
	\clearpage
\textcolor{white}{.}

\vspace{.1in}
	
	{\footnotesize
		\begin{table}[!ht]
			\begin{tabular}{|c|c|c|}
				\hline 
				$Q$ & $\dim_{\kk} \kk Q$ & $\dim_{\kk} \hay $\\\hline\hline
				$\begin{tikzcd}[column sep=scriptsize, row sep= small]
					\bullet_1 \arrow[r] & \bullet_2 \arrow[r] & \cdots \arrow[r] & \bullet_n
				\end{tikzcd} \quad n\geq 1$ & 
				$\frac{n(n+1)}{2}$ & 
				$\frac{n(n+1)(2n+1)}{6}$ \\\hline
				$\begin{gathered}\begin{tikzcd}[column sep=scriptsize, row sep= small]
						\bullet_1 \arrow[r] & \bullet_2 & \arrow[l] \cdots & \arrow[l] \bullet \arrow[r] & \bullet_{2n}
					\end{tikzcd}\\
					\text{or}\\\begin{tikzcd}[column sep=scriptsize, row sep= small]
						\bullet_1 \arrow[r] & \bullet_2 & \arrow[l] \cdots \arrow[r] & \bullet & \arrow[l] \bullet_{2n-1}
				\end{tikzcd}\end{gathered} \quad n\geq 2$ & $2n-1$ & $2n^2-2n+1$ \\\hline
					$\begin{tikzcd}[column sep=scriptsize, row sep= small]
						\bullet_1 \arrow[rd] &\\
						& \bullet_3 \arrow[r] & \cdots \arrow[r] & \bullet_n\\
						\bullet_2 \arrow[ru] &
					\end{tikzcd}
				\quad n\geq 4$ & $\frac{n(n+1)}{2}-1 $ & $\frac{n(n+1)(2n+1)}{6}-1$  \\\hline
				$\begin{tikzcd}[column sep=scriptsize, row sep= small]
					\bullet_1  &\\
					& \bullet_3 \arrow[r] \arrow[lu] \arrow[ld] & \cdots \arrow[r] & \bullet_n\\
					\bullet_2  &
				\end{tikzcd} \quad n\geq 4$ & $\frac{n^2-3n+10}{2}$ & $\frac{2n^3-9n^2+61n-78}{6}$  \\\hline
			$\begin{tikzcd}[sep=scriptsize]
				&&\bullet&&\\
				\bullet \arrow[r]  & \bullet \arrow[r] & \bullet \arrow[u] \arrow[r] &  \bullet \arrow[r] & \bullet
			\end{tikzcd}$ & $19 $ & $87$  \\\hline
		$\begin{tikzcd}[sep=scriptsize]
			&&\bullet&&&\\
			\bullet \arrow[r]  & \bullet \arrow[r] & \bullet \arrow[u] \arrow[r] &  \bullet \arrow[r] & \bullet \arrow[r] & \bullet
		\end{tikzcd}$ & $25 $ & $131$  \\\hline
	$\begin{tikzcd}[sep=scriptsize]
		&&\bullet&&&&\\
		\bullet \arrow[r]  & \bullet \arrow[r] & \bullet \arrow[u] \arrow[r] &  \bullet \arrow[r] & \bullet \arrow[r] & \bullet \arrow[r] & \bullet
	\end{tikzcd}$ & $32 $ & $188$  \\\hline
			\end{tabular}
				\vspace{.2in}
			
			\caption{$\dim_\kk \kk Q$ and $\dim_\kk \hay$ for some $Q$ of Dynkin type ADE}  \label{tab:dimADE}
		\end{table}
	}
	

	\bibliographystyle{alpha}
	\bibliography{biblio}
	
\end{document}